\documentclass[a4paper,11pt,reqno]{amsart}

\usepackage{amssymb}
\usepackage[all]{xy}

\usepackage{amsmath}
\usepackage{amsthm}

\usepackage{hyperref}
\usepackage{xcolor}
\usepackage{enumitem}
\usepackage{subcaption}
\usepackage[final]{graphicx}
\usepackage{float}
\usepackage{epic}
\usepackage{setspace}

\usepackage[all]{xy}
\usepackage{color}
\usepackage{calligra}
\usepackage[T1]{fontenc}

\usepackage{verbatim}

\newcount\mycount

\numberwithin{equation}{section}

\numberwithin{equation}{section}
\newtheorem{introtheorem}{Theorem}

\newtheorem{introdefinition}[introtheorem]{Definition}

\newtheorem{theorem}{Theorem}[section]

\newtheorem{lemma}[theorem]{Lemma}

\newtheorem{proposition}[theorem]{Proposition}

\newtheorem*{theorem*}{Theorem}

\theoremstyle{definition}
\newtheorem{definition}[theorem]{Definition}
\newtheorem*{definition*}{Definition}

\theoremstyle{remark}

\theoremstyle{remark}

\newtheorem{remark}[theorem]{Remark}

\newcommand{\cG}{\mathcal{G}}
\newcommand{\cR}{\mathcal{R}}
\newcommand{\cB}{\mathcal{B}}

\newcommand{\cS}{\mathcal{S}}

\newcommand{\op}{\text{op}}

\newcommand{\zero}{^{(0)}}

\DeclareMathOperator*{\esssup}{ess\,sup}

\usepackage[backend=biber,sorting=nyt, maxalphanames=6, maxbibnames=6, uniquelist=false,style=alphabetic,useprefix=true]{biblatex}
\DeclareFieldFormat[misc]{title}{\mkbibquote{#1}}

\begin{document}

\title{Kesten's criterion for discrete probability measure-preserving groupoids}

\author{Soham Chakraborty}
\author{Milan Donvil}
\author{Felipe Flores}
\author{Mario Klisse}

\address{Départment de Mathématiques et Applications, École Normale Supérieure, 45 Rue d'Ulm, 75005 Paris, France}

\email{soham.chakraborty@ens.psl.eu}

\email{milan.donvil@ens.psl.eu}

\address{Department of Mathematics, University of Virginia, Kerchof Hall 114. 141 Cabell Dr, Charlottesville, Virginia, United States}

\email{hmy3tf@virginia.edu}

\address{Department of Mathematics, Christian-Albrechts University Kiel,\\ Heinrich-Hecht-Platz 6,
Kiel, Germany}

\email{klisse@math.uni-kiel.de}

\date{\today. }

\begin{abstract}
Inspired by Kesten’s criterion for the amenability of groups, we establish a characterization of the amenability of discrete probability measure-preserving groupoids in terms of the operator norms of symmetric invariant Markov operators.  
\end{abstract}

\maketitle


\section{Introduction}

\vspace{3mm}

Amenability of groups, introduced by von Neumann in \cite{Neumann1929}, is one of the central notions in modern group theory, probability theory, ergodic theory, and functional analysis. A group is called \emph{amenable} if it possesses an invariant mean. This concept admits numerous equivalent characterizations; for instance, via the existence of \emph{Følner sequences} (see \cite{Folner1955}), the \emph{nuclearity} of the group's reduced group C$^{\ast}$-algebra (see \cite{Lance1973}), the existence of approximately invariant sequences of probability measures (see \cite{Reiter1968}), or the existence of fixed points for continuous affine actions on compact convex subsets of locally convex topological vector spaces (see \cite{Day1956}). For further characterizations, see \cite{Pier1984}.

In his seminal work \cite{Kesten1959a} (see also \cite{Kesten1959b}), Kesten studied random walks on Cayley graphs of finitely generated groups and obtained yet another characterization of amenability, relating it to the decay of return probabilities to the identity. More precisely, if the support of a given random walk generates the group, then the probability of returning to the identity in $2n$ steps decays exponentially if and only if the group is non-amenable.

Kesten's criterion has been proven useful in several contexts, see e.g. \cite{Kesten1959b,KaimanovichVershik83,Ol05,ST10, BV25}. It can be formulated naturally in terms of the spectral radius of the Markov operator associated with the random walk, in a way that we explicitly cite now.

\begin{theorem*}[Kesten's criterion \cite{Kesten1959a}] \label{KestenCriterion} Let $G$ be a countable discrete group, let $\mu$ be a symmetric probability measure on $G$ whose support generates $G$, and let $P_{\mu}$ be the associated Markov operator on $\ell^{2}(G)$ defined by $(P_{\mu}\xi)(g) := \sum_{h\in G} \xi(gh)\,\mu(h)$ for $\xi\in\ell^{2}(G)$,  $g\in G$. Then $G$ is amenable if and only if the spectral radius of $P_{\mu}$ is equal to $1$. \end{theorem*}

Over time, amenability has found natural extensions in many other areas of mathematics. In particular, Zimmer introduced an analog for discrete group actions and countable measured equivalence relations in \cite{Zimmer1977a, Zimmer1977b, Zimmer1978} in the late seventies; soon after, Renault extended the notion to general measured groupoids \cite{Renault1980} (see also \cite{Anantharaman-DelarocheRenault2000, Corlette2004}).

As in the case of groups, Renault’s amenability for measured groupoids admits several equivalent formulations (see, e.g., \cite{Anantharaman-DelarocheRenault2000}). In \cite{Ka05}, Kaimanovich introduced the \emph{fiberwise Liouville property} for measured groupoids, formulated in terms of invariant Markov operators acting fiberwise with respect to a Haar system. A measured groupoid equipped with such an invariant Markov operator is called \emph{fiberwise Liouville} if almost all fiberwise actions admit no non-trivial bounded harmonic functions. A strengthened version of this notion, adapting the classical Choquet--Deny property to the discrete measured groupoid setting, was developed in \cite{Berendschot2024}. Kaimanovich showed that a measured groupoid which admits an invariant fiberwise Liouville Markov operator is amenable and conjectured that the reverse implication should hold as well, which was confirmed by Chu and Li in \cite{Li2018} (see also \cite{BuKa21}). This yields a groupoid analog of the characterization of amenability of groups in terms of bounded harmonic functions, proven in one direction and conjectured by Furstenberg \cite{Furstenberg73} and proven in the other direction independently by  Kaimanovich–Vershik \cite{VershikKaimanovich79, KaimanovichVershik83} and Rosenblatt \cite{Ros81}.

Analogs of Kesten’s criterion have been formulated in a variety of contexts, including invariant random subgroups \cite{AGV14}, group extensions of topological Markov chains \cite{Stadlbauer2013}, and quantum groups \cite{Banica1999}. Building on Kaimanovich’s framework of invariant Markov operators on groupoids, the goal of this article is to establish an analog of Kesten’s criterion in the setting of discrete probability measure-preserving measured groupoids. To this end, denote by $P^\pi$ the Markov operator associated with a Borel field of probability measures $\pi$ (see Subsection \ref{MarkovDefinition}). We introduce the following definition.

\begin{introdefinition}[see Definition \ref{KestenDefinition}]\label{maindefinition} Let $(\mathcal{G},\mu)$ be a discrete probability measure-preserving groupoid. We say that $(\mathcal{G},\mu)$ satisfies \emph{Kesten’s criterion} if, for every symmetric Borel field of probability measures $\pi$, the restriction of the Markov operator $P^\pi$ to any invariant Borel subset $E \subseteq \mathcal{G}^{(0)}$ with $\mu(E)>0$ has operator norm equal to $1$.
\end{introdefinition}

Whereas the original Kesten's criterion stated above shows that for groups it suffices that a single Markov operator associated with a symmetric non-degenerate measure has norm one in order to deduce amenability, this is no longer true for general measured groupoids, as observed in \cite{Kai01}. Instead, we obtain the following characterization, whose proof relies on Hayes' approach in \cite{Ha24} (see also \cite{AFH24}). 

\begin{introtheorem}\label{MainTheorem} 
Let $(\mathcal{G},\mu)$ be a discrete probability measure-preserving groupoid. Then $(\mathcal{G},\mu)$ is amenable if and only if it satisfies Kesten’s criterion. \end{introtheorem}

Our theorem fits into the broader program of finding connections between probability theory and the study of measured equivalence relations and measured groupoids. In that same program, one finds the already mentioned works of Kaimanovich \cite{Kai01,Ka05}, Hayes \cite{Ha24}, and Abert, Fraczyk and Hayes \cite{AFH24}. These works make explicit use of equivalence relations. Note, however, that the study of random unimodular random graphs, trees, or surfaces requires the study of (often covertly defined) particular measured equivalence relations (see \cite{AlLy07,AbBi22,AFH25}). These objects arise naturally in the study of percolation theory \cite{BeLySc15}.

We also deduce the following interpretation of the norm of $P^{\pi}$ associated with a symmetric Borel field of probability measures $\pi$, and hence of the amenability of the ambient groupoid. In particular, we provide a formula involving averages of return probabilities (see Remark \ref{formula}), which recovers part of Kesten's original inspiration.

\begin{introtheorem} \label{MainTheorem2}
Let $(\mathcal{G},\mu)$ be a discrete probability measure-preserving groupoid, and let $\pi$ be a symmetric Borel field of probability measures on $\mathcal{G}$. Then the \emph{$E$-spectral radius} of $\pi$,
\begin{equation*}
\rho_{E}(\mathcal{G},\mu) := \lim_{n \to \infty} \left( \mu(E)^{-1} \left\langle (P^{\pi})^{2n} \chi_{E}, \chi_{E} \right\rangle \right)^{\frac{1}{2n}}\,,
\end{equation*}
exists and satisfies $\rho_{E}(\mathcal{G},\pi) \leq 1$ for every Borel subset $E \subseteq \mathcal{G}^{(0)}$ with $\mu(E)>0$. Moreover, $\rho_{\mathcal{G}^{(0)}}(\mathcal{G},\pi) = \Vert P^{\pi} \Vert$.
\end{introtheorem}

\vspace{3mm}

\noindent \emph{Structure of the article}. The paper is organized as follows. Section~2 collects the necessary background on functional analysis, invariant means and almost invariant vectors for measure-preserving transformations, as well as on measured groupoids. In Section~3, we introduce and study invariant Markov operators on groupoids induced by symmetric Borel fields of probability measures and establish our main results: Theorem~\ref{MainTheorem} and Theorem~\ref{MainTheorem2}. Finally, we have an appendix that contains examples illustrating that non-symmetric Borel fields of probability measures may give rise to unbounded Markov operators, hence justifying the assumption of symmetry.

\vspace{3mm}


\section{Preliminaries and notation}


\vspace{3mm}

\subsection{General notation} We denote by $\mathbb{N}:=\{0,1,2 \ldots \}$ the set of non-negative integers, by $\mathbb{N}_{\geq 1}:=\{1,2, \ldots \}$ the set of positive integers, and $[n]:=\{1,\ldots,n\}$ for $n \in \mathbb{N}_{\geq 1}$. For a set $S$ we write $\#S$ for the number of
elements in $S$ and $\chi_S$ for the characteristic function on $S$. The \emph{symmetric difference} of two sets $A$ and $B$ is denoted by $A \bigtriangleup B :=(A\setminus B)\cup (B\setminus A)$.


\vspace{3mm}

\subsection{Functional-analytic preliminaries} \label{subsec:Functional-analytic-preliminaries}

For Banach spaces $\mathcal{X}$ and $\mathcal{Y}$, we denote the Banach space of all bounded linear operators $T:\mathcal{X}\to\mathcal{Y}$, equipped with the operator norm, by $\mathbb{B}(\mathcal{X},\mathcal{Y})$. In the case where $\mathcal{X}=\mathcal{Y}$, we abbreviate $\mathbb{B}(\mathcal{X}):=\mathbb{B}(\mathcal{X},\mathcal{X})$.

The \emph{dual space} $\mathcal{X}^{*}:=\mathbb{B}(\mathcal{X},\mathbb{C})$ of $\mathcal{X}$ carries a natural locally convex topology weaker than the norm topology, called the \emph{weak$^{*}$-topology}, which is generated by the subbase of sets of the form
\begin{equation*}
\left\{\psi\in\mathcal{X}^{*}\mid \, |\phi(v)-\psi(v)|<\varepsilon \text{ for every } v\in F\right\},
\end{equation*}
where $\phi \in \mathcal{X}^\ast$, $\varepsilon > 0$, and where $F\subseteq \mathcal{X}$ is finite. By the \emph{Banach--Alaoglu theorem}, the closed unit ball $\{\phi\in\mathcal{X}^{*}\mid \|\phi\|\leq 1\}$ of $\mathcal{X}^\ast$ is compact with respect to this topology.

We will also make use of the \emph{spectral theorem} for bounded self-adjoint operators on Hilbert spaces. Let $\mathcal{H}$ be a Hilbert space and $T \in \mathbb{B}(\mathcal{H})$ a self-adjoint operator. Then every pair of vectors $\xi,\eta\in\mathcal{H}$ admits a complex Borel measure $\mu_{\xi,\eta}$ on the spectrum $\sigma(T)\subseteq \mathbb{R}$ of $T$ such that
\begin{equation} \label{SpectralTheorem}
\langle T\xi,\eta\rangle
= \int_{\sigma(T)} t \, d\mu_{\xi,\eta}(t)\,.
\end{equation}
If $\|\xi\|=1$, the measure $\mu_{\xi,\xi}$ is a probability measure.

Since $T$ is assumed to be self-adjoint, the norm closure $C^{\ast}(1,T)$ of the span of all elements $T^{k}$ with $k\in\mathbb{N}$ is a unital closed subalgebra of $\mathbb{B}(\mathcal{H})$, which is invariant under taking the adjoint in $\mathbb{B}(\mathcal{H})$. The \emph{continuous functional calculus} asserts that there exists a unital isometric isomorphism $C(\sigma(T))\rightarrow C^{\ast}(1,T)$, $f\mapsto f(T)$ of algebras which maps $\text{id}_{\sigma(T)}\in C(\sigma(T))$ to $T$. Here $C(\sigma(T))$ denotes the complex-valued continuous functions on $\sigma(T)$, equipped with the supremum norm. For $f\in C(\sigma(T))$ the identity in \eqref{SpectralTheorem} then extends via
\begin{equation*}
\langle f(T)\xi,\eta\rangle=\int_{\sigma(T)} f(t)\, d\mu_{\xi,\eta}(t)
\quad \text{for all }\xi,\eta\in\mathcal{H}\,.
\end{equation*}

Further details on these constructions can be found in Conway's classical text \cite{Co90}.

\vspace{3mm}


\subsection{Invariant means and almost invariant vectors}\label{group actions}

A \emph{standard measure space} is a pair $(Y,\nu)$ where $Y$ is a standard Borel space and $\nu$ is a $\sigma$-finite Borel measure on $Y$. If $\nu$ is a probability measure, then $(Y,\nu)$ is called a \emph{standard probability space}.

Let $\Gamma$ be a group, and let $(Y,\nu)$ be a $\sigma$-finite measure space equipped with a right action $\Gamma \curvearrowright (Y,\nu)$ by measure-preserving transformations. For $p \in [1,\infty]$ and $\gamma \in \Gamma$, we define a bounded linear operator $\alpha_{p}(\gamma) \in \mathbb{B}(L^{p}(Y,\nu))$ by
\begin{equation*}
(\alpha_{p}(\gamma) f)(x) = f(x\gamma)
\qquad\text{for } f \in L^{p}(Y,\nu),\ x \in Y\,.
\end{equation*}
By abuse of notation, we will usually suppress the parameter $p$ and simply write $\alpha := \alpha_{p}$.

A functional $m\in L^{\infty}(Y,\nu)^{*}$ is called a \emph{mean} on $Y$ if $m(1)=1$ and $m(f)\ge 0$ whenever $f\ge 0$. Since the set of all means on $Y$ is a weak$^{*}$-closed subset of the unit ball of $L^{\infty}(Y,\nu)^{*}$, it is weak$^{*}$-compact. A mean $m \in L^{\infty}(Y,\nu)^{*}$ is called \emph{$\Gamma$-invariant} if $m(\alpha(\gamma)f) = m(f)$ for all $f \in L^{\infty}(Y,\nu)$, $\gamma \in \Gamma$.

In the next section, we will use several standard characterizations of the existence of almost invariant vectors and the connection with invariant means found in \cite[Theorem~2.6, Theorem~2.7, Theorem~2.8]{Ha24}.

\vspace{3mm}


\subsection{Measured groupoids}\label{pmp}

For a groupoid $\mathcal{G}$ we denote its \emph{unit space} by $\mathcal{G}^{(0)}$, and we write $s$ and $t$ for the \emph{source} and \emph{target} maps, respectively. The set of \emph{composable pairs} of $\mathcal{G}$ is denoted by $\mathcal{G}^{(2)}:=\{(g,h)\in\mathcal{G}\times\mathcal{G}\mid s(g)=t(h)\}$, while $g^{-1}$ denotes the inverse of an element $g\in\mathcal{G}$.

For any subset $E \subseteq \mathcal{G}^{(0)}$ we set
\begin{equation*}
\mathcal{G}_{E} := s^{-1}(E), \quad \mathcal{G}^{E} := t^{-1}(E), \quad 
\mathcal{G}_{E}^{E} := \mathcal{G}_{E} \cap \mathcal{G}^{E}\,,
\end{equation*}
and for $x \in \mathcal{G}^{(0)}$ we abbreviate $\mathcal{G}_{x} := \mathcal{G}_{\{x\}}$, $\mathcal{G}^{x} := \mathcal{G}^{\{x\}}$, and $\mathcal{G}_{x}^{x} := \mathcal{G}_{\{x\}}^{\{x\}}$.  Furthermore, for subsets $A,B\subseteq\mathcal{G}$ we define $A B$ to be the set of products $ab$ with $a \in A$ and $b\in B$ for which $s(a)=t(b)$, which may be empty.

A \emph{discrete Borel groupoid} is a groupoid $\mathcal{G}$ endowed with the structure of a standard Borel space such that the unit space $\mathcal{G}^{(0)}$ is a Borel subset of $\mathcal{G}$, the maps $s$, $t$, multiplication, and inversion are Borel measurable functions, and all source and target fibres are at most countable (i.e.\ $s$ and $t$ are countable-to-one). Given such a groupoid and a Borel probability measure $\mu$ on $\mathcal{G}^{(0)}$, we define measures $\mu_s$ and $\mu_t$ on $\mathcal{G}$ by
\begin{equation*}
    \mu_{s}(A) := \int_{\mathcal{G}^{(0)}} \#(\mathcal{G}_{x} \cap A)\, d\mu(x)\,, \qquad \mu_{t}(A) := \int_{\mathcal{G}^{(0)}} \#(\mathcal{G}^{x} \cap A)\, d\mu(x)\,,
\end{equation*}
for every Borel set $A \subseteq \mathcal{G}$.  
We say that the pair $(\mathcal{G},\mu)$ is a \emph{discrete probability measure-preserving} (\emph{discrete p.m.p.}) groupoid if $\mu_s = \mu_t$; in this case we also say that $\mathcal{G}$ \emph{preserves} the measure $\mu$.

Let $(\mathcal{G},\mu)$ be a discrete p.m.p.\ groupoid. A Borel subset $E \subseteq \mathcal{G}^{(0)}$ is called \emph{invariant} if $\mu(t(\mathcal{G} E) \triangle E) = 0$.  
The pair $(\mathcal{G},\mu)$ is called \emph{ergodic} if every invariant Borel set $E \subseteq \mathcal{G}^{(0)}$ satisfies $\mu(E) \in \{0,1\}$.

A Borel subset $\gamma \subseteq \mathcal{G}$ is a (Borel) \emph{bisection} if for every $x \in \mathcal{G}^{(0)}$ both sets $\gamma\{x\}$ and $\{x\}\gamma$ contain at most one element. The collection of all Borel bisections forms an inverse semigroup under the product $(\gamma, \gamma') \mapsto \gamma\gamma'$.  
The \emph{full group} $[\mathcal{G}]$ of $\mathcal{G}$ consists of all Borel bisections $\gamma$ satisfying $\gamma\gamma^{-1} = \gamma^{-1}\gamma = \mathcal{G}^{(0)}$.  
For $\gamma \in [\mathcal{G}]$ and $x \in \mathcal{G}^{(0)}$ we identify the singleton sets $\gamma x$ and $x \gamma$ with the unique elements they contain. With this convention, for any Borel subset $E \subseteq \mathcal{G}^{(0)}$ with $\mu(E) > 0$ we obtain a unitary representation $\alpha_{E} : [\mathcal{G}] \longrightarrow \mathcal{U}\big(L^{2}(\mathcal{G}|_{E}, \mu_{E})\big)$ via
\begin{equation*}
(\alpha_{E}(\gamma)\xi)(g) := \xi(g\gamma) \; \text{ for } \xi \in L^{2}(\mathcal{G}|_{E},\mu_{E}), \gamma \in [\mathcal{G}], g \in \mathcal{G}|_{E}\,,
\end{equation*}
where the \emph{restriction} $\mathcal{G}|_{E}$ of $\mathcal{G}$ to $E$ is the discrete p.m.p.\ groupoid obtained by equipping $\mathcal{G}_{E}^{E}$ with the normalized measure $\mu_{E} := \mu(E)^{-1}\mu|_{E}$.   This construction  is a special case of the representation introduced in Subsection~\ref{group actions}. Indeed, one easily checks that $[\mathcal{G}]$ acts on $\mathcal{G}|_{E}$ by right multiplication in a measure-preserving way.  
When $E = \mathcal{G}^{(0)}$, we omit the index and simply write $\alpha = \alpha_{\mathcal{G}^{(0)}}$.


\vspace{3mm}

\section{Main results}

\vspace{3mm}

In this section we introduce invariant Markov operators on groupoids and establish our main results, Theorem~\ref{MainTheorem} and Theorem~\ref{MainTheorem2}. Throughout, the pair $(\mathcal{G},\mu)$ will always denote a discrete p.m.p.\ groupoid.



\subsection{Invariant Markov operators on groupoids} \label{MarkovDefinition}

Let $(\mathcal{G},\mu)$ be a discrete p.m.p.\ groupoid and let $\pi:\mathcal{G}\rightarrow\mathbb{C}$ be a Borel function. We call $\pi$ \emph{symmetric} if $\pi(g^{-1})=\overline{\pi(g)}$ for $\mu_{t}$-almost every $g\in\mathcal{G}$. Moreover, $\pi$ is called a (Borel) \emph{field of probability measures} if $\pi(g)\geq 0$ for $\mu_{t}$-almost every $g\in\mathcal{G}$ and $\sum_{g\in\mathcal{G}^{x}} \pi(g)=1$ for $\mu$-almost every $x\in\mathcal{G}^{(0)}$.

Following Hahn \cite{Ha78}, the \emph{$I$-norm} of a Borel function $\pi:\mathcal{G}\to\mathbb{C}$ is defined by
\begin{equation*}
\|\pi\|_{I}
=\max\Big\{
\esssup_{x\in\mathcal{G}^{(0)}} \sum_{g\in\mathcal{G}_{x}} |\pi(g)|,\;
\esssup_{x\in\mathcal{G}^{(0)}} \sum_{g\in\mathcal{G}^{x}} |\pi(g)|
\Big\}\,.
\end{equation*}
If $\pi$ is symmetric, then $\|\pi\|_{I}= \esssup_{x\in\mathcal{G}^{(0)}} \sum_{g\in\mathcal{G}^{x}} |\pi(g)|$, so in particular every symmetric field of probability measures has $I$-norm equal to $1$.

Whenever $\pi$ has finite $I$-norm, it induces for each $p\in[1,\infty]$ a bounded operator $P^{\pi}_{p}\in\mathbb{B}(L^{p}(\mathcal{G},\mu_t))$, as established in the following lemma. As before, we usually omit the index $p$ and write $P^{\pi}=P^{\pi}_{p}$. For $p=2$, the operator $P^{\pi}$ will be referred to as the \emph{invariant Markov operator} associated with $\pi$. Note that Kaimanovich focuses his work on the case $p=\infty$ \cite{Ka05}.

\begin{lemma}\label{I-lemma}
Let $(\mathcal{G},\mu)$ be a discrete p.m.p.\ groupoid, let $p\in[1,\infty]$, and let $\pi:\mathcal{G}\rightarrow\mathbb{C}$ be a Borel function with $\|\pi\|_{I}<\infty$. Then the map $P^{\pi}:L^{p}(\mathcal{G},\mu_{t})\rightarrow L^{p}(\mathcal{G},\mu_{t})$ given by $P^{\pi}(\xi)(g):= \sum_{h\in\mathcal{G}^{s(g)}} \xi(gh)\,\pi(h)$ for every $\xi\in L^{p}(\mathcal{G},\mu_{t})$ and $\mu_{t}$-almost every $g\in\mathcal{G}$, is well-defined and defines an element of $\mathbb{B}(L^{p}(\mathcal{G},\mu_{t}))$ with operator norm at most $\|\pi\|_{I}$.
\end{lemma}

\begin{proof}
Let $p\in[1,\infty]$ and let $\pi:\mathcal{G}\rightarrow\mathbb{C}$ be a Borel function with $\Vert\pi\Vert_{I}<\infty$. We distinguish three cases. 
\begin{itemize}
\item \emph{Case 1}: Consider the case $p=1$ and let $\xi\in L^{1}(\mathcal{G},\mu_{t})$. Then, $\sum_{h\in\mathcal{G}^{r(g)}} |\xi(h)|\, |\pi(g^{-1}h)|\le \|\pi\|_{I} \,\|\xi|_{\mathcal{G}^{r(g)}}\|_{1}$ for $\mu_{t}$-almost every $g\in\mathcal{G}$. Thus the sum defining $P^{\pi}(\xi)$ converges absolutely, and, arguing as in case 3, we get that $\|P^{\pi}\|\le\|\pi\|_{I}$.

\item \emph{Case 2}: For $p=\infty$ and $\xi\in L^{\infty}(\mathcal{G},\mu_{t})$ one has that
\begin{equation*}
\sum_{h\in\mathcal{G}^{s(g)}} |\xi(gh)|\,|\pi(h)|
\le \|\xi\|_{\infty} \sum_{h\in\mathcal{G}^{s(g)}} |\pi(h)|
\le \|\pi\|_{I} \|\xi\|_{\infty}
\end{equation*}
for $\mu_{t}$-almost every $g \in \mathcal{G}$. As above, $P^{\pi}$ is well defined and $\|P^{\pi}\|\le\|\pi\|_{I}$.

\item \emph{Case 3}: Now assume that $1<p<\infty$ and let $q\in(1,\infty)$ with $\frac{1}{p}+\frac{1}{q}=1$. With Hölder's inequality it follows that for $\xi\in L^{p}(\mathcal{G},\mu_{t})$, 
\begin{eqnarray}
\nonumber
\sum_{h\in\mathcal{G}^{s(g)}}|\xi(gh)||\pi(h)| &\leq& \left(\sum_{h\in\mathcal{G}^{s(g)}}|\xi(gh)|^{p}|\pi(h)|\right)^{\frac{1}{p}}\left(\sum_{h\in\mathcal{G}^{s(g)}}|\pi(h)|\right)^{\frac{1}{q}} \\
&\leq& \|\pi\|_{I}^{\frac{1}{q}}\left(\sum_{h\in\mathcal{G}^{s(g)}}|\xi(gh)|^{p}|\pi(h)|\right)^{\frac{1}{p}}\label{eq:Well-defined}
\end{eqnarray}
for $\mu_{t}$-almost every $g\in\mathcal{G}$. By $\sum_{h\in\mathcal{G}^{r(g)}}|\xi(h)|^{p}|\pi(g^{-1}h)| \leq \Vert\pi\Vert_{I} \,\Vert\xi|_{\mathcal{G}^{s(g)}}\Vert_{p}^{p}$ the sum on the right-hand side of \eqref{eq:Well-defined} converges absolutely. This in particular implies that $P^{\pi}(\xi)$ is well-defined with
\begin{eqnarray*}
\|P^{\pi}(\xi)\|_{p}^{p} &=& \int_{\mathcal{G}}\left|\sum_{h\in\mathcal{G}^{s(g)}}\xi(gh)\pi(h)\right|^{p}d\mu_{t}(g) \\
&=& \int_{\mathcal{G}^{(0)}}\left(\sum_{g\in\mathcal{G}^{x}}\left|\sum_{h\in\mathcal{G}^{s(g)}}\xi(gh)\pi(h)\right|^{p}\right)d\mu(x) \\
&\leq& \|\pi\|_{I}^{p} \, \|\xi\|_{p}^{p}.
\end{eqnarray*}
\end{itemize}
This finishes the proof.
\end{proof}

From the lemma it in particular follows that every symmetric Borel field of probability measures gives rise to a Markov operator of norm at most $1$.

Given Borel functions $\pi_{1}$ and $\pi_{2}$ on a discrete p.m.p.\ groupoid $(\mathcal{G},\mu)$ with $\|\pi_{1}\|_{I}, \|\pi_{2}\|_{I}<\infty$, their \emph{convolution product} is defined by
\begin{equation*}
(\pi_{1}\ast\pi_{2})(g)
:= \sum_{h\in\mathcal{G}^{r(g)}} \pi_{1}(h)\,\pi_{2}(h^{-1}g),
\; \text{ for } g\in\mathcal{G}\,.
\end{equation*}
This produces again a Borel function with finite $I$-norm. The convolution is associative with $P^{\pi_{1}\ast\pi_{2}} = P^{\pi_{1}} P^{\pi_{2}}$ and $\chi_{\mathcal{G}^{(0)}} \ast \pi = \pi = \pi \ast \chi_{\mathcal{G}^{(0)}}$.

Furthermore, Borel functions on groupoids admit a natural \emph{involution}: for a Borel function $\pi: \mathcal{G} \rightarrow \mathbb{C}$ the function $\pi^\ast: \mathcal{G}\rightarrow \mathbb{C}, g\mapsto \overline{\pi(g^{-1})}$ is again Borel.

\vspace{3mm}


\subsection{Amenable groupoids}

As mentioned earlier, Renault's notion of amenability for (discrete) measured groupoids admits several equivalent characterizations (see, e.g., \cite{Anantharaman-DelarocheRenault2000}). For our purposes, the following formulation will be the most convenient.

\begin{definition}[{\cite[Definition 3.2.8]{Anantharaman-DelarocheRenault2000}}]
Let $(\mathcal{G},\mu)$ be a discrete measured groupoid. We say that $(\mathcal{G},\mu)$ is \emph{amenable} if there exists a bounded linear map $\Phi\colon L^{\infty}(\mathcal{G},\mu_{t})\to L^{\infty}(\mathcal{G}^{(0)},\mu)$ of norm $1$ such that $\Phi(f)=f$ for all $f\in L^{\infty}(\mathcal{G}^{(0)},\mu)$, and $\Phi(\alpha(\gamma)(f))=\alpha(\gamma)(\Phi(f))$ for all $f\in L^{\infty}(\mathcal{G},\mu_{t})$ and all $\gamma\in[\mathcal{G}]$. The map $\Phi$ is called a \emph{global invariant mean}.
\end{definition}

The proof of the ``if'' direction of Theorem \ref{MainTheorem} requires the following lemma.

\begin{lemma}\label{reduction}
Let $(\mathcal{G},\mu)$ be a discrete p.m.p.\ groupoid. Then $(\mathcal{G},\mu)$ is amenable if and only if there exists a $[\mathcal{G}]$-invariant mean $m\in L^{\infty}(\mathcal{G},\mu_{t})^{\ast}$ such that $m(f)=\int_{\mathcal{G}^{(0)}} f\, d\mu$ for all $f\in L^{\infty}(\mathcal{G}^{(0)},\mu)$.
\end{lemma}

\begin{proof}
For the ``only if'' direction, assume that there exists a global invariant mean $\Phi\colon L^{\infty}(\mathcal{G},\mu_{t})\to L^{\infty}(\mathcal{G}^{(0)},\mu)$ and define $m\in L^{\infty}(\mathcal{G},\mu_{t})^{\ast}$ by $m(f):=\int_{\mathcal{G}^{(0)}} \Phi(f)\, d\mu$ for $f\in L^{\infty}(\mathcal{G},\mu_{t})$. Then $m$ is a $[\mathcal{G}]$-invariant mean, and $m(f)=\int_{\mathcal{G}^{(0)}} f\, d\mu$ for all $f\in L^{\infty}(\mathcal{G}^{(0)},\mu)$.

For the ``if'' direction, let $m\in L^{\infty}(\mathcal{G},\mu_{t})^{\ast}$ be a $[\mathcal{G}]$-invariant mean satisfying $m(f)=\int_{\mathcal{G}^{(0)}} f\, d\mu$ for $f\in L^{\infty}(\mathcal{G}^{(0)},\mu)$. Fix $f\in L^{\infty}(\mathcal{G},\mu_{t})$. For every $k\in L^{\infty}(\mathcal{G}^{(0)},\mu)$ we have
\begin{equation*}
|m(k f)|\leq m(|f k|)\leq \|f\|_{\infty} m(|k|)=\|f\|_{\infty}\|k\|_{1}\,.
\end{equation*}
Since $\mu$ is a probability measure, $L^{\infty}(\mathcal{G}^{(0)},\mu)$ is dense in $L^{1}(\mathcal{G}^{(0)},\mu)$. Hence there exists a unique functional $L_{f}\in L^{1}(\mathcal{G}^{(0)},\mu)^{\ast}\cong L^{\infty}(\mathcal{G}^{(0)},\mu)$ with $\|L_{f}\|\leq \|f\|_{\infty}$ such that $L_{f}(k)=m(fk)$ for all $k\in L^{\infty}(\mathcal{G}^{(0)},\mu)$. Let $\Phi(f)\in L^{\infty}(\mathcal{G}^{(0)},\mu)$ be the function corresponding to $L_{f}$ so that $m(fk)=\int_{\mathcal{G}^{(0)}} \Phi(f) k\, d\mu$ for all $k\in L^{1}(\mathcal{G}^{(0)},\mu)$.

Positivity of $m$ shows that $\Phi$ is positive and has norm $1$. Moreover, since $m(f)=\int_{\mathcal{G}^{(0)}} f\, d\mu$ for all $f\in L^{\infty}(\mathcal{G}^{(0)},\mu)$, we have $\Phi(f)=f$ on $L^{\infty}(\mathcal{G}^{(0)},\mu)$. Finally,
\begin{equation*}
\int_{\cG\zero} \Phi(f)\, d\mu = m(f)=m(\alpha(\gamma)f)=\int_{\cG\zero} \Phi(\alpha(\gamma)f)\, d\mu
\end{equation*}
for all $f \in L^{\infty}(\mathcal{G}^{(0)},\mu)$ and all $\gamma\in[\mathcal{G}]$, proving that $\Phi$ is a global invariant mean.
\end{proof}

Lemma \ref{reduction} shows that to prove amenability of a discrete p.m.p.\ groupoid $(\mathcal{G},\mu)$ it suffices to construct a $[\mathcal{G}]$-invariant mean on $L^{\infty}(\mathcal{G},\mu_{t})$. In the context of Theorem \ref{MainTheorem} we achieve this by producing $\Gamma$-invariant means for arbitrary countable subgroups $\Gamma\leq[\mathcal{G}]$, and then applying a compactness argument based on the weak$^*$-topology. This strategy is heavily inspired by Hayes' approach in \cite{Ha24}.

The following proposition states a classical characterization of amenability which we will use in the ``only if'' direction of the proof of Theorem \ref{MainTheorem}, often referred to as the \emph{weak Godement condition}; see, for example, \cite[Definition~7.1]{Ana11}.

\begin{proposition}\label{amAD}
A discrete measured groupoid $(\mathcal{G},\mu)$ is amenable if and only if there exists a sequence $(\xi_{n})_{n\in\mathbb{N}}$ of Borel functions on $\mathcal{G}$ such that:
\begin{enumerate}
\item[(i)] $\sum_{g\in\mathcal{G}_{x}} |\xi_{n}(g)|^{2}=1$ for $\mu$-almost every $x\in\mathcal{G}^{(0)}$ and every $n\in\mathbb{N}$;
\item[(ii)] $F_{n}\to 1$ in the weak$^{\ast}$-topology on $L^{\infty}(\mathcal{G},\mu_{t})\cong L^{1}(\mathcal{G},\mu_{t})^{\ast}$, where $F_{n}\in L^{\infty}(\mathcal{G},\mu_{t})$ is given by  $F_{n}(h):=\sum_{g\in\mathcal{G}_{t(h)}} \xi_{n}(gh)\,\overline{\xi_{n}(g)}$ for $h\in\mathcal{G}$.
\end{enumerate}
\end{proposition}

\vspace{3mm}


\subsection{Proof of Kesten's criterion for amenability}

We now work with Markov operators on restrictions of groupoids. For a discrete p.m.p.\ groupoid $(\mathcal{G},\mu)$, an invariant Borel subset $E\subseteq\mathcal{G}^{(0)}$ with $\mu(E)>0$, and a Borel function $\pi\colon\mathcal{G}\to\mathbb{C}$ with $\|\pi\|_{I}<\infty$, we denote by $P_{E}^{\pi}\in\mathbb{B}(L^{2}(\mathcal{G}|_{E},(\mu_{E})_{t}))$ the Markov operator associated with $(\mathcal{G}|_{E},\mu_{E})$ and the restriction $\pi|_{\mathcal{G}|_{E}}$.

Definition \ref{maindefinition} can then be formulated as follows.

\begin{definition}\label{KestenDefinition}
Let $(\mathcal{G},\mu)$ be a discrete p.m.p.\ groupoid. We say that $(\mathcal{G},\mu)$ satisfies \emph{Kesten's criterion} if, for every invariant Borel subset $E\subseteq\mathcal{G}^{(0)}$ with $\mu(E)>0$ and every symmetric Borel field of probability measures $\pi\colon\mathcal{G}\to\mathbb{C}$, one has $\|P_{E}^{\pi}\|=1$.
\end{definition}

In the next lemma, a countable subgroup $\Gamma\leq[\mathcal{G}]$ is said to \emph{cover} $\mathcal{G}$ if the union $\bigcup_{\gamma\in\Gamma}\gamma$ is a co-null set in $\cG$. 

\begin{lemma}\label{PIF}
Let $(\mathcal{G},\mu)$ be a discrete p.m.p.\ groupoid such that $\|P_{E}^{\pi}\|=1$ for every invariant Borel subset $E\subseteq\mathcal{G}^{(0)}$ with $\mu(E)>0$ and every symmetric field $\pi\colon\mathcal{G}\to\mathbb{C}$ of probability measures. Then for any invariant Borel subsets $E_{1},\ldots,E_{k}\subseteq\mathcal{G}^{(0)}$ of positive measure and any countable subgroup $\Gamma\leq[\mathcal{G}]$ covering $\mathcal{G}$, there exists a $\Gamma$-invariant mean $m\in L^{\infty}(\mathcal{G},\mu)^{\ast}$ satisfying $m(\chi_{E_{i}})=\mu(E_{i})$ for $1\leq i\leq k$.
\end{lemma}

\begin{proof}
Let $\mathcal{F}$ denote the $\sigma$-algebra generated by $E_{1},\ldots,E_{k}$, and let $\nu\in\mathrm{Prob}(\Gamma)$ be a symmetric probability measure whose support generates $\Gamma$. There exists a finite partition $(A_{j})_{1\leq j\leq l}$ of $\mathcal{G}^{(0)}$ into invariant Borel subsets of positive measure such that $\mathcal{F}=\bigl\{\bigcup_{j\in D} A_{j} \mid D\subseteq[l]\bigr\}$.

Define $\pi\colon\mathcal{G}\to[0,\infty)$ by $\pi(g):=\nu(\{\gamma\in\Gamma \mid g\in\gamma\})$. A simple computation shows that $\pi$ is a symmetric Borel field of probability measures. Similarly,
\begin{align*}
(P_{A_{j}}^{\pi}\xi)(g)
&=\sum_{h\in\mathcal{G}^{s(g)}} \xi(gh)\, \nu(\{\gamma \in \Gamma \mid  h\in\gamma\}) \\
&=\sum_{\gamma\in\Gamma} \xi(g\gamma) \, \nu(\gamma) \\
&=\left( \alpha_{A_{j}}(\nu)(\xi)\right) (g)\,.
\end{align*}
for all $1\leq j\leq l$, $\xi\in L^{2}(\mathcal{G}|_{A_{j}},(\mu_{A_{j}})_{t})$, $g\in\mathcal{G}|_{A_{j}}$ where $\alpha_{A_j}(\nu):=\sum_{\gamma \in \Gamma} \nu(\gamma)\alpha_{A_j}(\gamma)$, so that $P_{A_{j}}^{\pi}=\alpha_{A_{j}}(\nu)$. By assumption, $\|\alpha_{A_{j}}(\nu)\|=\|P_{A_{j}}^{\pi}\|=1$, so \cite[Theorem~2.6]{Ha24} and \cite[Theorem~2.7]{Ha24} yield a $\Gamma$-invariant mean $m_{j}\in L^{\infty}(\mathcal{G}|_{A_{j}},(\mu_{A_{j}})_{t})^{\ast}$.

Define $m\in L^{\infty}(\mathcal{G},\mu_{t})^{\ast}$ by $m(f):=\sum_{j=1}^{l} \mu(A_{j}) \, m_{j}(f|_{\mathcal{G}|_{A_{j}}})$ for $f \in L^{\infty}(\mathcal{G},\mu_{t})$. Since the $A_{j}$ are invariant and disjoint, $m$ is $\Gamma$-invariant and satisfies $m(\chi_{A_{j}})=\mu(A_{j})$  for all $1\leq j\leq l$, hence $m(\chi_{E_{i}})=\mu(E_{i})$ for all $i$. The result then follows from Theorem \cite[Theorem~2.8]{Ha24}.
\end{proof}

\begin{proof}[Proof of Theorem \ref{MainTheorem}]
For the ``if'' direction assume that the discrete p.m.p. groupoid $(\mathcal{G},\mu)$ satisfies Kesten's criterion. For a countable subgroup $\Gamma\leq[\mathcal{G}]$ and a finite tuple $\mathbf{E}=(E_{1},\ldots,E_{k})$ of invariant Borel subsets of positive measure, let $\mathcal{M}_{\Gamma,\mathbf{E}}$ denote the set of $\Gamma$-invariant means $m\in L^{\infty}(\mathcal{G},\mu_{t})^{\ast}$ with $m(\chi_{E_{i}})=\mu(E_{i})$ for $1\leq i \leq k$. By Lemma \ref{PIF}, each $\mathcal{M}_{\Gamma,\mathbf{E}}$ is non-empty and weak$^{\ast}$-compact.

By the finite intersection property the intersection $\bigcap_{\Gamma,\mathbf{E}} \mathcal{M}_{\Gamma,\mathbf{E}}$ is non-empty. Hence there exists a $[\mathcal{G}]$-invariant mean $m \in L^{\infty}(\mathcal{G},\mu_{t})^{*}$ satisfying $m(\chi_{E})=\mu(E)$ for every invariant Borel subset $E\subseteq\mathcal{G}^{(0)}$. By \cite[Theorem~2.8]{Ha24}, $m(f)=\int_{\mathcal{G}^{(0)}} f\, d\mu$ for all $f\in L^{\infty}(\mathcal{G}^{(0)},\mu)$, and Lemma \ref{reduction} then implies that $(\mathcal{G},\mu)$ is amenable.

For the ``only if'' direction assume that $(\mathcal{G},\mu)$ is amenable. Let $\pi\colon\mathcal{G}\to [0,\infty)$ be a symmetric Borel field of probability measures and let $E\subseteq\mathcal{G}^{(0)}$ be an invariant Borel subset with $\mu(E)>0$. Lemma \ref{I-lemma} shows that $\|P_{E}^{\pi}\|\leq1$.

To show the reverse inequality, let $(\xi_{n})_{n \in \mathbb{N}}$ be sequence of Borel functions as in Proposition \ref{amAD}, and for each $n$ let $\xi_{n}^{E}\in L^{2}(\mathcal{G}|_{E},(\mu_{E})_{t})$ be the restriction of $\xi_{n}$ to $\mathcal{G}|_{E}$. Then,
\begin{align*}
\langle P_{E}^{\pi}\xi_{n}^{E},\xi_{n}^{E}\rangle & =\frac{1}{\mu(E)}\int_{E}\sum_{g\in\mathcal{G}_{x}}(P_{E}^{\pi}\xi_{n}^{E})(g)\overline{\xi_{n}^{E}(g)}d\mu(x)\\
 & =\frac{1}{\mu(E)}\int_{E}\sum_{g\in\mathcal{G}_{x}}\sum_{h\in\mathcal{G}^{x}}\xi_{n}(gh)\overline{\xi_{n}(g)}\pi(h)d\mu(x)\\
 & =\frac{1}{\mu(E)}\int_{E}\sum_{h\in\mathcal{G}^{x}}\pi(h)\sum_{g\in\mathcal{G}_{x}}\xi_{n}(gh)\overline{\xi_{n}(g)}d\mu(x) \\
 & \rightarrow\frac{1}{\mu(E)}\int_{E}\sum_{h\in\mathcal{G}^{x}}\pi(h)d\mu(x) \\
  & =1\,.
\end{align*}
Hence $\|P_{E}^{\pi}\|=1$, completing the proof.
\end{proof}

\begin{remark}
For groups, it suffices that one Markov operator associated with a symmetric non-degenerate measure has norm $1$ in order to conclude amenability. As shown by the examples in \cite{Kai01}, this is no longer true at the level of general groupoids. However, for a measurable bundle of groups such a reduction still works. Indeed, if a Markov operator $P$ has norm $1$, there exists a sequence $(\xi_{n})_{n \in \mathbb{N}}\subseteq L^{2}(\mathcal{G},\mu_{t})$ of positive unit vectors such that $\langle P\xi_{n},\xi_{n}\rangle\to 1$. A computation analogous to that in the proof of the ``only if'' direction of Theorem \ref{MainTheorem} shows that the $\xi_{n}$ are asymptotically invariant. Moreover, the $\xi_{n}$ may be chosen so that the sum over the target fibres equals $1$. Since in a bundle of groups the source and target fibres coincide, Proposition \ref{amAD} then implies that the bundle is amenable.
\end{remark}

\vspace{3mm}


\subsection{A formula for the spectral radius}

\emergencystretch3em
In \cite{Kesten1959a, Kesten1959b}, Kesten characterized the amenability of finitely generated groups in terms of the return probabilities of random walks on their Cayley graphs. Motivated by this perspective, we introduce the following notion.

\begin{definition}
Let $(\mathcal{G},\mu)$ be a discrete p.m.p.\ groupoid. Given a Borel subset $E\subseteq\mathcal{G}^{(0)}$ with $\mu(E)>0$ and a symmetric Borel field of probability measures $\pi:\mathcal{G}\rightarrow [0,\infty)$, we define
\begin{equation}
\rho_{E}(\mathcal{G},\pi)
:=\lim_{n\to\infty}\mu(E)^{-\frac{1}{2n}}
\left\langle (P^{\pi})^{2n}\chi_{E},\chi_{E}\right\rangle^{\frac{1}{2n}}
\label{eq:SpectralRadius}
\end{equation}
and call it the \emph{$E$-spectral radius of $\pi$}. When $E=\mathcal{G}^{(0)}$, we write simply $\rho(\mathcal{G},\pi):=\rho_{\mathcal{G}^{(0)}}(\mathcal{G},\pi)$ and refer to this as the \emph{spectral radius of $\pi$}.
\end{definition}

Recall that Theorem~\ref{MainTheorem2} asserts that the limit in \eqref{eq:SpectralRadius} always exists, satisfies $\rho_{E}(\mathcal{G},\pi)\le 1$, and that $\rho(\mathcal{G},\pi)=\|P^{\pi}\|$. Before proving this, we record two observations.

\begin{remark}\label{formula}
\emph{(i)} Let $(\mathcal{G},\mu)$ be a discrete p.m.p.\ groupoid, $E\subseteq\mathcal{G}^{(0)}$ a Borel subset with $\mu(E)>0$, and $\pi:\mathcal{G}\to [0,\infty)$ a symmetric Borel field of probability measures. For $n\in\mathbb{N}_{\ge 1}$ denote by $\pi^{*n}$ the $n$-fold convolution power of $\pi$. Then
\begin{equation*}
\bigl((P^{\pi})^{n}\chi_{E}\bigr)(g)
=
\begin{cases}
\pi^{*n}(g^{-1})\,, & t(g)\in E\,,\\[2mm]
0\,, & \text{otherwise},
\end{cases}
\end{equation*}
for every $g\in\mathcal{G}$. Consequently,
\begin{equation*}
\langle (P^{\pi})^{n}\chi_{E},\chi_{E}\rangle
=\int_{E}\pi^{*n}(x)\,d\mu(x)
\end{equation*}
and hence
\begin{equation*}
\rho_{E}(\mathcal{G},\pi)
=\lim_{n\to\infty}\mu(E)^{-\frac{1}{2n}}
\left(\int_{E}\pi^{*(2n)}(x)\,d\mu(x)\right)^{\frac{1}{2n}}\,.
\end{equation*}
Following the framework of Kaimanovich in \cite{Ka05}, the quantity $\mu(E)^{-1}\int_{E}\pi^{*(2n)}(x)\,d\mu(x)$
may be interpreted as the normalized $\mu$-average of the probability of returning to the set $E$ after $2n$ steps.

\emph{(ii)} For a discrete p.m.p.\ groupoid  $(\mathcal{G},\mu)$ the same computation used in the previous remark shows that $P^{\pi}(\chi_{\mathcal{G}^{(0)}})=\overline{\pi^{*}}$ for every Borel function $\pi:\mathcal{G}\to\mathbb{C}$ with $\|\pi\|_{I}<\infty$. Hence, appealing to Lemma~\ref{I-lemma},
\begin{equation*}
\|\pi\|_{2}
=\|P^{\pi}(\chi_{\mathcal{G}^{(0)}})\|_{2}
\le \|P^{\pi}\|
\le \|\pi\|_{I}\,.
\end{equation*}
Therefore, if an operator $T\in\mathbb{B}(L^{2}(\mathcal{G},\mu_{t}))$ can be written as a norm limit $ T=\lim_{i\to\infty}P^{\pi_{i}}$ for suitable Borel functions $\pi_{i}:\mathcal{G}\to\mathbb{C}$ with finite $I$-norms, then $T$ itself has to be of the form $T=P^\eta$, for some Borel function $\eta\in L^2(\cG,\mu_t)$. In this case, $\eta$ can be retrieved as $\eta=\overline{T(\chi_{\mathcal{G}^{(0)}})^\ast}$ and it follows that $ T(\chi_{\mathcal{G}^{(0)}})$ is non-zero whenever $T\neq 0$.
\end{remark}

We now proceed to the proof of Theorem~\ref{MainTheorem2}.

\begin{proof}[Proof of Theorem~\ref{MainTheorem2}]
Since $\pi$ is symmetric, $P^{\pi}$ is self-adjoint. Indeed, for $\xi,\eta\in L^{2}(\mathcal{G},\mu_{t})$,
\begin{eqnarray*}
\langle P^{\pi}\xi,\eta\rangle &=& \int_{\mathcal{G}^{(0)}}\sum_{g\in\mathcal{G}_{x}}(P^{\pi}\xi)(g)\overline{\eta(g)}d\mu(x) =\int_{\mathcal{G}^{(0)}}\sum_{g\in\mathcal{G}_{x}}\sum_{h\in\mathcal{G}^{x}}\xi(gh)\overline{\eta(g)}\pi(h)d\mu(x)\\
 &=& \int_{\mathcal{G}^{(0)}}\sum_{h\in\mathcal{G}^{x}}\pi(h)\sum_{g\in\mathcal{G}_{x}}\xi(gh)\overline{\eta(g)}d\mu(x) =\int_{\mathcal{G}^{(0)}}\sum_{h\in\mathcal{G}_{x}}\pi(h^{-1})\sum_{g\in\mathcal{G}_{x}}\xi(gh^{-1})\overline{\eta(g)}d\mu(x)\\
 &=& \int_{\mathcal{G}^{(0)}}\sum_{h\in\mathcal{G}_{x}}\pi(h)\sum_{g\in\mathcal{G}_{t(h)}}\xi(g)\overline{\eta(gh)}d\mu(x) =\int_{\mathcal{G}^{(0)}}\sum_{g\in\mathcal{G}_{x}}\xi(g)\sum_{h\in\mathcal{G}^{x}}\overline{\eta(gh)}\pi(h)d\mu(x)\\
 & =& \int_{\mathcal{G}^{(0)}}\sum_{g\in\mathcal{G}_{x}}\xi(g)\overline{(P^{\pi}\eta)(g)}d\mu(x)=\langle\xi,P^{\pi}\eta\rangle\,.
\end{eqnarray*}

Fix a Borel subset $E\subseteq\mathcal{G}^{(0)}$ with $\mu(E)>0$ and define the unit vector $\xi_{E}:=\mu(E)^{-1/2}\chi_{E}\in L^{2}(\mathcal{G},\mu_{t})$. By the spectral theorem for self-adjoint operators (together with the discussion in Subsection~\ref{subsec:Functional-analytic-preliminaries}), there exists a probability measure $\nu_{E}$ on the spectrum $\sigma(P^{\pi})\subseteq[-1,1]$ such that
\begin{equation}
\langle f(P^{\pi})\xi_{E},\xi_{E}\rangle
=\int_{\sigma(P^{\pi})} f(t)\,d\nu_{E}(t)
\label{eq:SpectralTheorem}
\end{equation}
for every continuous function $f\in C(\sigma(P^{\pi}))$. In particular, the limit in \eqref{eq:SpectralRadius} exists and equals the $L^{\infty}$-norm of the function $t\mapsto |t|$ on $\sigma(P^{\pi})$.

Now assume that $E=\mathcal{G}^{(0)}$. Since $P^{\pi}$ is self-adjoint, there exists $\lambda\in\sigma(P^{\pi})$ with $\|P^{\pi}\|=|\lambda|$. We claim that $\lambda\in\operatorname{supp}(\nu_{E})$. Indeed, if we suppose otherwise, then some open neighborhood $U$ of $\lambda$ satisfies $\nu_{E}(U)=0$. Choose a non-negative $f\in C(\sigma(P^{\pi}))$ with $f(\lambda)=1$ and $f\equiv 0$ on $U^{c}\cap\sigma(P^{\pi})$. By \eqref{eq:SpectralTheorem},
\begin{equation*}
\langle f(P^{\pi})\xi_{E},\xi_{E}\rangle
=\int_{U\cap\sigma(P^{\pi})} f(t)\,d\nu_{E}(t)=0\,.
\end{equation*}
However, $\sqrt{f(P^{\pi})}$ is a non-zero positive operator in the smallest norm closed, self-adjoint subalgebra of $\mathbb{B}(L^{2}(\mathcal{G},\mu_{t}))$ containing $P^{\pi}$, and hence can be approximated in norm by operators of the form $P^{\pi_{i}}$, where the $\pi_{i}$ are Borel functions of finite $I$-norm. By Remark~\ref{formula} (ii),
\begin{equation*}
\langle f(P^{\pi})\xi_{E},\xi_{E}\rangle
=\|\sqrt{f(P^{\pi})}\xi_{E}\|_{2}^{2}\neq 0\,,
\end{equation*}
a contradiction. Thus $\lambda\in\operatorname{supp}(\nu_{E})$.

Combining this with the previous paragraph yields $\rho(\mathcal{G},\pi)=|\lambda|=\|P^{\pi}\|$, as desired.
\end{proof}


\vspace{3mm}

\section{Appendix: unbounded Markov operators}

\vspace{3mm}

The purpose of this appendix is to illustrate how Borel fields of probability measures with infinite $I$-norm can give rise to unbounded Markov operators. We have two examples. The first one occurs in a discrete probability measure-preserving non-ergodic groupoid, whereas the second one occurs in a measure-preserving discrete ergodic groupoid. 


\vspace{3mm}

\subsection{First example} 
We begin by recalling that any countable Borel equivalence relation gives rise to a discrete p.m.p.\ groupoid. More precisely, let $(X,\mu)$ be a standard probability space, and let $\mathcal R\subseteq X\times X$ be a countable Borel equivalence relation. Then $\mathcal R$ becomes a discrete measured groupoid with unit space $\mathcal R^{(0)} := {\sf Diag}(X) \cong X$ and structure maps
\begin{equation*}
 s(x,y) := (y,y)\,, \qquad t(x,y) := (x,x)\,, \qquad (x,y)(y,z) := (x,z)\,, \qquad (x,y)^{-1} := (y,x)
\end{equation*}
for all $(x,y),(y,z)\in\mathcal R$. We say that $\mathcal R$ is \emph{probability measure-preserving} (\emph{p.m.p.}) if the resulting groupoid is p.m.p.

If $X$ is countable and $\mathcal R = X\times X$ is the full equivalence relation, then a Borel system of probability measures reduces to a function $\pi\colon \mathcal R\to[0,\infty)$ such that $\sum_{x\in X} \pi(x,y)=1$ for every $y\in X$. The associated Markov operator satisfies
\begin{equation*}
 P^\pi(\xi)(x,y) = \sum_{z\in X} \xi\big((x,y)(y,z)\big)\pi(y,z)
 = \sum_{z\in X} \xi(x,z)\pi(y,z)\,,
\end{equation*}
for every $\xi\in L^2(\mathcal R,\mu_t)$; hence $P^\pi$ identifies with the (possibly infinite) matrix $[\pi(y,x)]_{x,y\in X}$.

\medskip

For what follows, fix a p.m.p.\ equivalence relation $\mathcal R_0$ on a standard probability space $(Y,\nu)$. For each $n\in\mathbb N_{\geq 1}$ let $\cS_n$ denote the full equivalence relation on $([n],\nu_n)$, where $\nu_n$ is the normalized counting measure. Recall also that if $(\cG_1,\mu_1)$ and $(\cG_2,\mu_2)$ are discrete p.m.p.\ groupoids, then their product $(\cG_1\times \cG_2,\mu_1\times \mu_2)$ is again a discrete p.m.p.\ groupoid with the obvious operations.

Consider now the equivalence relation $\mathcal R := \bigsqcup_{n\in\mathbb N} (\mathcal R_0\times \cS_n)$, defined on the space $ X := \bigsqcup_{n\in\mathbb N} (Y\times [n])$, and equipped with the measure $\mu := \sum_{n\in\mathbb N} (\nu\times \nu_n)$. With the obvious Borel structure, $\mathcal R$ becomes a discrete p.m.p.\ groupoid.

\medskip

Our construction of the Borel field $\pi$ of probability measures is based on direct sums of finite-dimensional operators (i.e., matrices) whose sequence of norms diverges. The following lemma provides the relevant building blocks.

\begin{lemma}\label{Prop: increasing norms of markov operators}
Let $n \in \mathbb{N}$ be an integer with $n>1$ and let $\delta\in(0,\tfrac12)$. Then there exists a matrix $A_\delta := [A_\delta(i,j)]_{1\le i,j\le n}\in \mathbb M_n(\mathbb C)$ with strictly positive entries and column sums $\sum_{i=1}^n A_\delta(i,j)=1$ for $1\leq j \leq n$ such that the corresponding operator norm satisfies $\|A_\delta\| > \sqrt{n}-\delta$.
\end{lemma}

\begin{proof}
For $\epsilon\in(0,1)$, let $B(\epsilon)$ be the matrix whose $n$ rows are identical and equal to
\begin{equation*}
 x_\epsilon := (1-\epsilon,\tfrac{\epsilon}{n-1},\ldots,\tfrac{\epsilon}{n-1})\,.
\end{equation*}
Writing $\mathbf{1}$ for the row vector of all ones, we have $B(\epsilon) = \mathbf{1}^T x_\epsilon$. Thus,
\begin{equation*}
 \|B(\epsilon)\| \geq \|\mathbf{1}\|_2 \,\|x_\epsilon\|_2
 = \sqrt{n}\,\sqrt{(1-\epsilon)^2 + \tfrac{\epsilon^2}{\,n-1\,}}\,.
\end{equation*}
The function $F(\epsilon) := (1-\epsilon)^2 + \frac{\epsilon^2}{n-1}$ is continuous on $[0,1]$ and attains its maximum at $\epsilon=0$. Hence, there exists a $\epsilon_0>0$ such that
\begin{equation*}
 \sqrt{F(\epsilon_0)} > 1 - \frac{\delta}{\sqrt{n}}\,.
\end{equation*}
Let $A_\delta:=B(\epsilon_0)$, so that $\|A_\delta\| > \sqrt{n}-\delta$.
\end{proof}

Fix $\delta\in(0,\tfrac12)$ and let $\pi_0' \colon \mathcal R_0\to[0,\infty)$ be a Borel field of probability measures. For $n\ge 1$, define a Borel field $\pi_n\colon \cS_n\to[0,\infty)$ of probability measures by setting $\pi_n(i,j):=A_\delta(i,j)$ for $1\le i,j\le n$, where $A_\delta$ is as in Lemma~\ref{Prop: increasing norms of markov operators}. These fields combine to a Borel field of probability measures $\pi: \mathcal{R} \rightarrow [0,1)$ defined to be equal to $\pi_0'\times \pi_n$ on each component $\mathcal R_0\times \cS_n$. The associated Markov operator decomposes as
\begin{equation*}
 P^\pi = \bigoplus_{n\in\mathbb N} (P^{\pi_0'} \otimes P^{\pi_n})
 \;\in\; \bigoplus_{n\in\mathbb N} \mathbb B\!\left(\ell^2(\mathcal R_0)\otimes \ell^2(\cS_n)\right)\,.
\end{equation*}
Since
\begin{equation*}
 \sup_{n\in\mathbb N} \|P^{\pi_0'}\otimes P^{\pi_n}\|
 = \|P^{\pi_0'}\|\left( \sup_{n\in\mathbb N}\|P^{\pi_n}\| \right)
 > \|P^{\pi_0'}\| \left( \sup_{n\in\mathbb N} (\sqrt{n}-\delta) \right)
 = \infty\,,
\end{equation*}
the operator $P^\pi$ is unbounded.

\medskip

The construction above enjoys several noteworthy features. First, if $(Y,\nu)$ is non-atomic, then so is the resulting unit space $(X,\mu)$. Second, if $\mathcal R_0$ is amenable, then $\mathcal R$ is amenable as well. Finally, the field $\pi$ is non-degenerate in the sense of \cite[Definition~2.1]{Berendschot2024}.


\vspace{3mm}

\subsection{Second example}

For our second example, consider the full equivalence relation $\mathcal R := \mathbb N\times \mathbb N$ on $\mathbb N$, equipped with the counting measure~$\mu$. The pair $(\mathcal R,\mu)$ forms a \emph{discrete measure-preserving} groupoid, meaning that the induced measures $\mu_s$ and $\mu_t$ introduced in Subsection~\ref{pmp} coincide; indeed, both agree with the counting measure on $\mathbb N^{2}$. Moreover, $(\mathcal R,\mu)$ is ergodic. Note, however, that $(\mathcal R,\mu)$ is not a discrete p.m.p.\ groupoid.

For $k\in\mathbb N$ define
\begin{equation*}
    I_k := \mathbb N \cap \big[\tfrac{k(k+1)}{2},\, \tfrac{(k+1)(k+2)}{2}\big)\,,
\end{equation*}
so that $\mathbb N = \bigsqcup_{k\in\mathbb N} I_k$ is a partition of the natural numbers. Define a map $\pi\colon \mathcal R\to[0,\infty)$ by
\begin{equation*}
 \pi(m,n) := 
 \begin{cases}
 1\,, & \text{if }n\in I_m\,,\\[2mm]
 0\,, & \text{otherwise},
 \end{cases}
\quad \text{for } (m,n)\in\mathcal R\,.
\end{equation*}
Then $\sum_{m\in\mathbb N}\pi(m,n)=1$ for each $n\in\mathbb N$, so $\pi$ is a (non-symmetric) Borel field of probability measures.

For each $k\in\mathbb N$ define $\xi_k\in L^2(\mathcal R,\mu_t)$ by
\begin{equation*}
 \xi_k(m,n) :=
 \begin{cases}
   \pi(k,n)\,, & \text{if } m=k\,,\\[1mm]
   0\,, & \text{otherwise},
 \end{cases}
\quad \text{for } (m,n)\in\mathcal R\,.
\end{equation*}
A direct computation shows that
\begin{equation*}
 P^\pi(\xi_k)(m,n)
   = \sum_{a\in\mathbb N} \xi_k(m,a)\,\pi(n,a)
   = 
   \begin{cases}
     \displaystyle \sum_{a\in\mathbb N}\pi(k,a)\pi(n,a)\,, & \text{if }m=k\,,\\[2mm]
     0\,, & \text{otherwise},
   \end{cases}
\end{equation*}
so that in particular,
\begin{equation*}
 P^\pi(\xi_k)(k,n)
   = \sum_{a\in\mathbb N}\pi(k,a)\pi(n,a)
   = \#(I_k\cap I_n)
   =
   \begin{cases}
     \#I_k\,, & \text{if }n=k\,,\\[1mm]
     0\,, & \text{otherwise}.
   \end{cases}
\end{equation*}
Hence,
\begin{equation*}
 \|P^\pi(\xi_k)\|_2^2
   = \sum_{m,n\in\mathbb N} |P^\pi(\xi_k)(m,n)|^2
   = \sum_{n\in\mathbb N} |P^\pi(\xi_k)(k,n)|^2
   = (\#I_k)^2
   = (k+1)^2\,,
\end{equation*}
while
\begin{equation*}
 \|\xi_k\|_2^2
   = \sum_{m,n\in\mathbb N} |\xi_k(m,n)|^2
   = \sum_{n\in\mathbb N} |\pi(k,n)|^2
   = \sum_{n\in I_k} 1
   = \#I_k
   = (k+1)\,.
\end{equation*}
Consequently,
\begin{equation*}
 \sup_{k\in\mathbb N} \frac{\|P^\pi(\xi_k)\|_2}{\|\xi_k\|_2}
   = \sup_{k\in\mathbb N} \sqrt{k+1}
   = \infty\,,
\end{equation*}
and therefore the Markov operator $P^\pi$ is not a bounded operator on $L^2(\mathcal R,\mu_t)$.



\vspace{3mm}

\section*{Acknowledgments}

\vspace{3mm}

S.C. and M.D. are supported by the ERC advanced grant 101141693. F.F. thankfully acknowledges support from the Simons Foundation Dissertation Fellowship SFI-MPS-SDF-00015100. He also thanks Professor Nicol\'as Matte Bon for the interesting discussions about Kesten's criterion. The authors would like to thank James Harbour for initial discussions and Ben Hayes for the discussions surrounding the cospectral radius on equivalence relations. Moreover, the authors also wish to thank Tey Berendschot for proposing the problem and the conferences ``YMC$^*$A'' at SDU and ``Orbit equivalence and topological and measurable dynamics'' at CIRM for hosting part of this research.


\vspace{3mm}

\emergencystretch3em
\printbibliography

@article {AFH25,
    AUTHOR = {Ab\'ert, M. and Fraczyk, M. and Hayes,
              B.},
     TITLE = {Growth dichotomy for unimodular random rooted trees},
   JOURNAL = {Ann. Probab.},
  FJOURNAL = {The Annals of Probability},
    VOLUME = {53},
      YEAR = {2025},
    NUMBER = {5},
     PAGES = {1627--1644},
}

@article {AlLy07,
    AUTHOR = {Aldous, D. and Lyons, R.},
     TITLE = {Processes on unimodular random networks},
   JOURNAL = {Electron. J. Probab.},
  FJOURNAL = {Electronic Journal of Probability},
    VOLUME = {12},
      YEAR = {2007},
    NUMBER = {54},
     PAGES = {1454--1508},
}

@article {AbBi22,
    AUTHOR = {Ab\'ert, M. and Biringer, I.},
     TITLE = {Unimodular measures on the space of all {R}iemannian
              manifolds},
   JOURNAL = {Geom. Topol.},
  FJOURNAL = {Geometry \& Topology},
    VOLUME = {26},
      YEAR = {2022},
    NUMBER = {5},
     PAGES = {2295--2404},
}

@misc{Ha24,
      TITLE={Coamenability and cospectral radius for orbit equivalence relations}, 
      author={B. Hayes},
      year={2024},
      archivePrefix={arXiv},  
      eprint={2410.16480},
      primaryClass={math.DS},
      
}

@article {AFH24,
    AUTHOR = {Abert, M. and Fraczyk, M. and Hayes, B.},
     TITLE = {Co-spectral radius for countable equivalence relations},
   JOURNAL = {Ergodic Theory Dynam. Systems},
  FJOURNAL = {Ergodic Theory and Dynamical Systems},
    VOLUME = {44},
      YEAR = {2024},
    NUMBER = {12},
     PAGES = {3385--3427},
}

@article {Ka05,
    AUTHOR = {Kaimanovich, V. A.},
     TITLE = {Amenability and the {L}iouville property},
   JOURNAL = {Israel J. Math.},
  FJOURNAL = {Israel Journal of Mathematics},
    VOLUME = {149},
      YEAR = {2005},
     PAGES = {45--85},
}

@incollection {BuKa21,
    AUTHOR = {B\"uhler, T. and Kaimanovich, V. A.},
     TITLE = {Amenability of groupoids and asymptotic invariance of
              convolution powers},
 BOOKTITLE = {Topology, geometry, and dynamics---{V}. {A}.
              {R}okhlin-{M}emorial},
    SERIES = {Contemp. Math.},
    NUMBER = {772},
     PAGES = {69--92},
 PUBLISHER = {Amer. Math. Soc., [Providence], RI},
      YEAR = {2021},
}

@article {AGV14,
    AUTHOR = {Ab\'ert, M. and Glasner, Y. and Vir\'ag, B.},
     TITLE = {Kesten's theorem for invariant random subgroups},
   JOURNAL = {Duke Math. J.},
  FJOURNAL = {Duke Mathematical Journal},
    VOLUME = {163},
      YEAR = {2014},
    NUMBER = {3},
     PAGES = {465--488},
}

@article {Kesten1959a,
    AUTHOR = {Kesten, H.},
     TITLE = {Symmetric random walks on groups},
   JOURNAL = {Trans. Amer. Math. Soc.},
  FJOURNAL = {Transactions of the American Mathematical Society},
    VOLUME = {92},
      YEAR = {1959},
     PAGES = {336--354},
    SORTTITLE = "a",
}

@article {Kesten1959b,
    AUTHOR = {Kesten, H.},
     TITLE = {Full {B}anach mean values on countable groups},
   JOURNAL = {Math. Scand.},
  FJOURNAL = {Mathematica Scandinavica},
    VOLUME = {7},
      YEAR = {1959},
     PAGES = {146--156},
    SORTTITLE = "b",
}

@article {Ana11,
    AUTHOR = {Anantharaman-Delaroche, C.},
    TITLE = {Old and new about treeability and the Haagerup property for measured groupoids},
    YEAR = {2011},
    NUMBER = {hal-00596887},
}

@incollection {Kai01,
    AUTHOR = {Kaimanovich, V. A.},
    TITLE = {Equivalence relations with amenable leaves need not be
              amenable},
    YEAR = {2001},
    BOOKTITLE = {Topology, ergodic theory, real algebraic geometry: Rokhlin’s Memorial},
    SERIES = {Amer. Math. Soc. Transl. Ser. 2},
    NUMBER = {202},
    PAGES = {151--166},
    PUBLISHER = {Amer. Math. Soc., Providence, RI},
}

@article{Neumann1929,
author = {von Neumann, J.},
journal = {Fundamenta Mathematicae},
number = {1},
pages = {73-116},
title = {Zur allgemeinen Theorie des Masses},
volume = {13},
year = {1929},
}

@article {Folner1955,
    AUTHOR = {F\o lner, E.},
     TITLE = {On groups with full {B}anach mean value},
   JOURNAL = {Math. Scand.},
  FJOURNAL = {Mathematica Scandinavica},
    VOLUME = {3},
      YEAR = {1955},
     PAGES = {243--254},
}

@article {Lance1973,
    AUTHOR = {Lance, C.},
     TITLE = {On nuclear {$C\sp{\ast} $}-algebras},
   JOURNAL = {J. Functional Analysis},
  FJOURNAL = {Journal of Functional Analysis},
    VOLUME = {12},
      YEAR = {1973},
     PAGES = {157--176},
}

@book {Reiter1968,
    AUTHOR = {Reiter, H.},
     TITLE = {Classical harmonic analysis and locally compact groups},
 PUBLISHER = {Clarendon Press, Oxford},
      YEAR = {1968},
}

@article {Day1956,
    AUTHOR = {Kaimanovich, V. A.},
     TITLE = {Amenable semigroups},
   JOURNAL = {Illinois J. Math.},
  FJOURNAL = {Illinois Journal of Mathematics},
    VOLUME = {1},
      YEAR = {1957},
     PAGES = {509--544},
}

@book {Pier1984,
    AUTHOR = {Pier, J.},
     TITLE = {Amenable locally compact groups},
    SERIES = {Pure and Applied Mathematics (New York)},
      NOTE = {A Wiley-Interscience Publication},
 PUBLISHER = {John Wiley \& Sons, Inc., New York},
      YEAR = {1984},
}

@article {Zimmer1978,
    AUTHOR = {Zimmer, R. J.},
     TITLE = {Amenable ergodic group actions and an application to {P}oisson
              boundaries of random walks},
   JOURNAL = {J. Functional Analysis},
  FJOURNAL = {Journal of Functional Analysis},
    VOLUME = {27},
      YEAR = {1978},
    NUMBER = {3},
     PAGES = {350--372},
}

@article {Zimmer1977a,
    AUTHOR = {Zimmer, R. J.},
     TITLE = {Hyperfinite factors and amenable ergodic actions},
   JOURNAL = {Invent. Math.},
  FJOURNAL = {Inventiones Mathematicae},
    VOLUME = {41},
      YEAR = {1977},
    NUMBER = {1},
     PAGES = {23--31},
}

@article {Zimmer1977b,
    AUTHOR = {Zimmer, R. J.},
     TITLE = {On the von {N}eumann algebra of an ergodic group action},
   JOURNAL = {Proc. Amer. Math. Soc.},
  FJOURNAL = {Proceedings of the American Mathematical Society},
    VOLUME = {66},
      YEAR = {1977},
    NUMBER = {2},
     PAGES = {289--293},
}

@book {Co90,
    AUTHOR = {Conway, J. B.},
     TITLE = {A course in functional analysis},
    SERIES = {Graduate Texts in Mathematics},
    NUMBER = {96},
   EDITION = {Second},
 PUBLISHER = {Springer-Verlag, New York},
      YEAR = {1990},
}

@article {Ha78,
    AUTHOR = {Hahn, P.},
     TITLE = {The regular representations of measure groupoids},
   JOURNAL = {Trans. Amer. Math. Soc.},
  FJOURNAL = {Transactions of the American Mathematical Society},
    VOLUME = {242},
      YEAR = {1978},
     PAGES = {35--72},
}

@book {Renault1980,
    AUTHOR = {Renault, J.},
     TITLE = {A groupoid approach to {$C\sp{\ast} $}-algebras},
    SERIES = {Lecture Notes in Mathematics},
    NUMBER = {793},
 PUBLISHER = {Springer, Berlin},
      YEAR = {1980},
}

@book {Anantharaman-DelarocheRenault2000,
    AUTHOR = {Anantharaman-Delaroche, C. and Renault, J.},
     TITLE = {Amenable groupoids},
    SERIES = {Monographies de L'Enseignement Math\'ematique},
    NUMBER = {36},
 PUBLISHER = {L'Enseignement Math\'ematique, Geneva},
      YEAR = {2000},
     PAGES = {196},
}

@article {Corlette2004,
    AUTHOR = {Corlette, K. and Hern\'andez Lamoneda, L. and Iozzi,
              A.},
     TITLE = {A vanishing theorem for the tangential de {R}ham cohomology of
              a foliation with amenable fundamental groupoid},
   JOURNAL = {Geom. Dedicata},
  FJOURNAL = {Geometriae Dedicata},
    VOLUME = {103},
      YEAR = {2004},
     PAGES = {205--223},
}

@misc{Berendschot2024,
  title={The Choquet-Deny Property for Groupoids}, 
      author={T. Berendschot and S. Chakraborty and M. Donvil and S.-J. Kim and M. Klisse},
      year={2024},
      eprint={2406.05004},
      archivePrefix={arXiv},
      primaryClass={math.FA},
}

@article {Li2018,
    AUTHOR = {Chu, C. and Li, X.},
     TITLE = {Amenability, {R}eiter's condition and {L}iouville property},
   JOURNAL = {J. Funct. Anal.},
  FJOURNAL = {Journal of Functional Analysis},
    VOLUME = {274},
      YEAR = {2018},
    NUMBER = {12},
     PAGES = {3291--3324},
}

@article {VershikKaimanovich79,
    AUTHOR = {Kaimanovich, V. A. and Vershik, A. M.},
     TITLE = {Random walks on groups: boundary, entropy, uniform
              distribution},
   JOURNAL = {Dokl. Akad. Nauk SSSR},
  FJOURNAL = {Doklady Akademii Nauk SSSR},
    VOLUME = {249},
      YEAR = {1979},
    NUMBER = {1},
     PAGES = {15--18},
}

@article {KaimanovichVershik83,
    AUTHOR = {Kaimanovich, V. A. and Vershik, A. M.},
     TITLE = {Random walks on discrete groups: boundary and entropy},
   JOURNAL = {Ann. Probab.},
  FJOURNAL = {The Annals of Probability},
    VOLUME = {11},
      YEAR = {1983},
    NUMBER = {3},
     PAGES = {457--490},
}

@incollection {Furstenberg73,
    AUTHOR = {Furstenberg, H.},
     TITLE = {Boundary theory and stochastic processes on homogeneous
              spaces},
 BOOKTITLE = {Harmonic analysis on homogeneous spaces ({P}roc. {S}ympos.
              {P}ure {M}ath., {V}ol. {XXVI}, {W}illiams {C}oll.,
              {W}illiamstown, {M}ass., 1972)},
    SERIES = {Proc. Sympos. Pure Math.},
     PAGES = {193--229},
 PUBLISHER = {Amer. Math. Soc., Providence, RI},
      YEAR = {1973},
}

@article {Stadlbauer2013,
    AUTHOR = {Stadlbauer, M.},
     TITLE = {An extension of {K}esten's criterion for amenability to
              topological {M}arkov chains},
   JOURNAL = {Adv. Math.},
  FJOURNAL = {Advances in Mathematics},
    VOLUME = {235},
      YEAR = {2013},
     PAGES = {450--468},
}

@article {Banica1999,
    AUTHOR = {Banica, T.},
     TITLE = {Representations of compact quantum groups and subfactors},
   JOURNAL = {J. Reine Angew. Math.},
  FJOURNAL = {Journal f\"ur die Reine und Angewandte Mathematik. [Crelle's
              Journal]},
    VOLUME = {509},
      YEAR = {1999},
     PAGES = {167--198},
}

@article {BeLySc15,
    AUTHOR = {Benjamini, I. and Lyons, R. and Schramm, O.},
     TITLE = {Unimodular random trees},
   JOURNAL = {Ergodic Theory Dynam. Systems},
  FJOURNAL = {Ergodic Theory and Dynamical Systems},
    VOLUME = {35},
      YEAR = {2015},
    NUMBER = {2},
     PAGES = {359--373},
}

@article {Ros81,
    AUTHOR = {Rosenblatt, J.},
    TITLE = {Ergodic and mixing random walks on locally compact groups},
    JOURNAL = {Math. Ann.},
    VOLUME = {257},
    YEAR = {1981},
    PAGES = {31--42},
}

@book {Ol05,
    AUTHOR = {Ollivier, Y.},
     TITLE = {A {J}anuary 2005 invitation to random groups},
    SERIES = {Ensaios Matem\'aticos [Mathematical Surveys]},
    NUMBER = {10},
 PUBLISHER = {Sociedade Brasileira de Matem\'atica, Rio de Janeiro},
      YEAR = {2005},
}

@article{ST10,
    AUTHOR = {Shalom, Y. and Tao, T.},
    TITLE = {A finitary version of Gromov's polynomial growth theorem},
    JOURNAL = {Geom. Funct. Anal.},
    YEAR = {2010},
    PAGES = {1502--1547},
}

@article{BV25,
    author = {Berendschot, T. and Vaes, S.},
    title = {Measure equivalence embeddings of free groups and free group factors},
    journal = {Ann. Scient. \'Ec. Norm. Sup.},
    year = {2025},
    PAGES = {389--418},
}

\end{document}